\documentclass[11pt,twoside]{article}
\usepackage{mathrsfs}
\usepackage{amssymb}
\usepackage{amsmath}
\usepackage[all]{xy}
\usepackage{amsthm}
\numberwithin{equation}{section}

\newtheorem{theorem}{Theorem}[section]
\newtheorem{proposition}[theorem]{Proposition}
\newtheorem{definition}[theorem]{Definition}
\newtheorem{remark}[theorem]{Remark}
\newtheorem{lemma}[theorem]{Lemma}
\newtheorem{example}[theorem]{Example}

\newtheorem{corollary}[theorem]{Corollary}

\newcommand{\edge}{\ar@{-}}

\newcommand{\pf}{\noindent\begin {proof}}
\newcommand{\epf}{\end{proof}}

\newcommand{\Hom}{\mbox{\rm Hom}}

\def\Im{\mathop{\rm Im}\nolimits}
\def\Ker{\mathop{\rm Ker}\nolimits}

\def\mod{\mathop{\rm mod}\nolimits}

\def\id{\mathop{\rm id}\nolimits}
\def\pd{\mathop{\rm pd}\nolimits}
\def\max{\mathop{\rm max}\nolimits}
\def\min{\mathop{\rm min}\nolimits}
\def\sup{\mathop{\rm sup}\nolimits}
\def\inf{\mathop{\rm inf}\nolimits}
\def\add{\mathop{\rm add}\nolimits}

\def\gldim{\mathop{\rm gl.dim}\nolimits}

\def\rad{\mathop{{\rm rad}}\nolimits}
\def\soc{\mathop{{\rm soc}}\nolimits}
\def\top{\mathop{{\rm top}}\nolimits}

\def\dim{\mathop{\rm dim}\nolimits}

\def\Hom{\mathop{\rm Hom}\nolimits}

\def\sup{\mathop{\rm sup}\nolimits}
\def\lim{\mathop{\underrightarrow{\rm lim}}\nolimits}

\def\gen{\mathop{\rm gen}\nolimits}

\def\proj{\mathop{\rm proj}\nolimits}

\def\End{\mathop{\rm End}\nolimits}

\def\C{\mathop{\rm \mathcal{C}}\nolimits}

\def\S{\mathop{\rm \mathcal{S}}\nolimits}

\def\V{\mathop{\rm \mathcal{V}}\nolimits}

\def\LL{\mathop{\rm LL}\nolimits}

\def\gen.dim{\mathop{\rm gen.dim}\nolimits}

\title{ \bf An Upper Bound for the Dimension of Bounded Derived Categories
}
\vspace{0.2cm}

\author { \ Junling  Zheng$^{a,b}$, Zhaoyong Huang$^{b,}$\thanks{Email: zjlshuxue@163.com, huangzy@nju.edu.cn}
\\
{\it \scriptsize  $^a$Department of Mathematics, China Jiliang University, Hangzhou 310018, Zhejiang Province, P.R. China
}\\{\it \scriptsize  $^b$Department of Mathematics, Nanjing University, Nanjing 210093, Jiangsu Province, P.R. China }}

\date{ }

\begin{document}

\baselineskip=16pt


\maketitle

\begin{abstract}
Let $\Lambda$ be an artin algebra. We give an upper bound for the dimension of the bounded derived category of
the category $\mod \Lambda$ of finitely generated right $\Lambda$-modules in terms of the projective and injective dimensions
of certain class of simple right $\Lambda$-modules as well as the radical layer length of $\Lambda$.
In addition, we give an upper bound for the dimension of the singularity category of
$\mod \Lambda$ in terms of the radical layer length of $\Lambda$.
\end{abstract}

\pagestyle{myheadings}
\markboth{\rightline {\scriptsize  }}
         {\leftline{\scriptsize  An upper bound for the dimension of derived category }}

\section{Introduction} 

Given a triangulated category $\mathcal{T}$, Rouquier introduced in \cite{Rou} the dimension $\dim\mathcal{T}$ of $\mathcal{T}$
under the idea of Bondal and van den Bergh in \cite{BvdB03G}. This dimension and the infimum of the Orlov spectrum of $\mathcal{T}$
coincide, see \cite{BFK12O, Orl09R}. Roughly speaking, it is an invariant that measures how quickly the category can be built from one object.
Many authors have studied the upper bound of $\dim \mathcal{T}$, see \cite{BFK12O, BOJ15T, CYZ, H, KrK, OpS12G, Rou2, Rou} and so on.
There are a lot of triangulated categories having infinite dimension,
for instance, Oppermann and \v St'ov\'\i \v cek proved in \cite{OpS12G} that
all proper thick subcategories of the bounded derived category of finitely generated modules over a Noetherian algebra
containing perfect complexes have infinite dimension.

Let $\Lambda$ be an artin algebra. Let $\mod \Lambda$ be the category of finitely generated right $\Lambda$-modules
and let $D^{b}(\mod \Lambda)$ and $D_{sg}^{b}(\mod \Lambda)$ be the bounded derived category and singularity category of
$\mod \Lambda$ respectively. The upper bounds for the dimensions of these two categories can be given in terms of
the Loewy length $\LL(\Lambda)$ and the global dimension $\gldim\Lambda$ of $\Lambda$.

\begin{theorem} \label{thm1.1}
Let $\Lambda$ be an artin algebra. Then we have
\begin{enumerate}
\item[(1)] {\rm (\cite[Proposition 7.37]{Rou})} $\dim D^{b}(\mod \Lambda) \leqslant \LL(\Lambda)-1;$
\item[(2)] {\rm (\cite[Proposition 7.4]{Rou} and \cite[Proposition 2.6]{KrK})} $\dim D^{b}(\mod \Lambda) \leqslant \gldim \Lambda;$
\item[(3)] {\rm (\cite[Lemma 4.5]{BOJ15T})} $\dim D_{sg}^{b}(\mod \Lambda) \leqslant \LL(\Lambda)-2$.
\end{enumerate}
\end{theorem}

By Theorem \ref{thm1.1}(1)(3), we have that both $\dim D^{b}(\mod \Lambda)$ and $\dim D_{sg}^{b}(\mod \Lambda)$ are finite;
however, Theorem \ref{thm1.1}(2) does not provide any information when $\gldim \Lambda$ is infinite.

For a length-category $\mathcal{C}$, generalizing the Loewy length, Huard, Lanzilotta and Hern\'andez introduced
in \cite{HLM,HLM2} the (radical) layer length associated with a torsion pair, which is a new measure for objects of
$\mathcal{C}$. Let $\Lambda$ be an artin algebra and $\mathcal{V}$ a set of some simple modules in $\mod \Lambda$.
Let $t_{\mathcal{V}}$ be the torsion radical of a torsion pair associated with $\mathcal{V}$ (see Section 3 for details). We use
$\ell\ell^{t_{\mathcal{V}}}(\Lambda)$ to denote the $t_{\mathcal{V}}$-radical layer length of $\Lambda$.
For a module $M$ in $\mod \Lambda$, we use $\pd M$ and $\id M$ to denote the projective and injective dimensions of $M$ respectively;
in particular, set $\pd M=-1=\id M$ if $M=0$. For a subclass $\mathcal{B}$ of $\mod \Lambda$, the {\bf projective dimension} $\pd\mathcal{B}$
of $\mathcal{B}$ is defined as
\begin{equation*}
\pd \mathcal{B}=
\begin{cases}
\sup\{\pd M\;|\; M\in \mathcal{B}\}, & \text{if} \;\; \mathcal{B}\neq \varnothing;\\
-1,&\text{if} \;\; \mathcal{B}=\varnothing.
\end{cases}
\end{equation*}
Dually, the {\bf injective dimension} $\id \mathcal{B}$ of $\mathcal{B}$ is defined. Note that $\mathcal{V}$ is a finite set. So, if each simple module
in $\mathcal{V}$ has finite projective (resp. injective) dimension, then $\pd \mathcal{V}$ (resp. $\id \mathcal{V}$) attains its (finite) maximum.

The aim of this paper is to prove the following

\begin{theorem}\label{thm1.2} {\rm (Theorems \ref{thm3.12} and \ref{thm3.14})}
Let $\Lambda$ be an artin algebra and $\mathcal{V}$ a set of some simple modules in $\mod \Lambda$
with $\ell\ell^{t_{\mathcal{V}}}(\Lambda)=n$. Then we have
\begin{enumerate}
\item[(1)] if $d=\min\{\pd\mathcal{V}, \id\mathcal{V}\}$,
then $\dim D^{b}(\mod \Lambda) \leqslant (d+2)(n+1)-2;$
\item[(2)] $\dim D_{sg}^{b}(\mod \Lambda)\leqslant \max\{0, n-2\}.$
\end{enumerate}
\end{theorem}

In Section 3, we give the proof of Theorem \ref{thm1.2}. In fact, Theorem \ref{thm1.1} is some special cases of Theorem \ref{thm1.2}
(see Remark \ref{rem3.16}). Moreover, by choosing some suitable $\mathcal{V}$ and applying Theorem \ref{thm1.2},
we may obtain more precise upper bounds for $\dim D^{b}(\mod \Lambda)$ and $\dim D^{b}_{sg}(\mod \Lambda)$
than that in Theorem \ref{thm1.1}. We give in Section 4 two examples to illustrate this and that the difference
between the upper bounds obtained from Theorems \ref{thm1.1} and \ref{thm1.2} may be arbitrarily large.

\section{Preliminaries}

\subsection{The dimension of a triangulated category}

We recall some notions from \cite{Opp07U,Rou2,Rou}.
Let $\mathcal{T}$ be a triangulated category and $\mathcal{I} \subseteq {\rm Ob}\mathcal{T}$.
Let $\langle \mathcal{I} \rangle$ be the full subcategory consisting of $\mathcal{T}$
of all direct summands of finite direct sums of shifts of objects in $\mathcal{I}$.
Given two subclasses $\mathcal{I}_{1}, \mathcal{I}_{2}\subseteq {\rm Ob}\mathcal{T}$, we denote $\mathcal{I}_{1}*\mathcal{I}_{2}$
by the full subcategory of all extensions between them, that is,
$$\mathcal{I}_{1}*\mathcal{I}_{2}=\{ X\mid  X_{1} \longrightarrow X \longrightarrow X_{2}\longrightarrow X_{1}[1]\;
{\rm with}\; X_{1}\in \mathcal{I}_{1}\; {\rm and}\; X_{2}\in \mathcal{I}_{2}\}.$$
Write $\mathcal{I}_{1}\diamond\mathcal{I}_{2}:=\langle\mathcal{I}_{1}*\mathcal{I}_{2} \rangle.$
Then $(\mathcal{I}_{1}\diamond\mathcal{I}_{2})\diamond\mathcal{I}_{3}=\mathcal{I}_{1}\diamond(\mathcal{I}_{2}\diamond\mathcal{I}_{3})$
for any subclasses $\mathcal{I}_{1}, \mathcal{I}_{2}$ and $\mathcal{I}_{3}$ of $\mathcal{T}$ by the octahedral axiom.
Write
\begin{align*}
\langle \mathcal{I} \rangle_{0}:=0,\;
\langle \mathcal{I} \rangle_{1}:=\langle \mathcal{I} \rangle\; {\rm and}\;
\langle \mathcal{I} \rangle_{n+1}:=\langle \mathcal{I} \rangle_{n}\diamond \langle \mathcal{I} \rangle_{1}\;{\rm for\; any \;}n\geqslant 1.
\end{align*}

\begin{definition}{\rm (\cite[Definiton 3.2]{Rou})}\label{def2.1}
The {\bf dimension} $\dim \mathcal{T}$ of a triangulated category $\mathcal{T}$
is the minimal $d$ such that there exists an object $M\in \mathcal{T}$ with
$\mathcal{T}=\langle M \rangle_{d+1}$. If no such $M$ exists for any $d$, then we set $\dim \mathcal{T}=\infty.$
\end{definition}

\begin{lemma}{\rm (\cite[Lemma 7.3]{Psa})}\label{lem2.2}
Let $\mathcal{T}$ be a triangulated category and let $X, Y$ be two objects of $\mathcal{T}$.
Then
$$\langle X \rangle _{m}\diamond \langle Y \rangle _{n} \subseteq \langle X\oplus Y \rangle _{m+n}$$
for any $m,n \geqslant 0$.
\end{lemma}

\begin{lemma}{\rm (\cite[Proposition 3.2]{AiT})}\label{lem2.3}
Let $\mathcal{A}$ be an abelian category admitting enough projective objects.
Let $X=(X^{i},d^{i})$ be a bounded complex in $\mathcal{A}$ such that
the homology $H^{i}(X)$ has projective dimension at most
$n$ for all $i$. Then $X\in\langle \mathcal{P}\rangle_{n+1}  \subseteq D^{b}(\mathcal{A})$
for the subcategory $ \mathcal{P} \subseteq \mathcal{A}$
of projective objects.
\end{lemma}

Dually, we have

\begin{lemma}\label{lem2.4}
Let $\mathcal{A}$ be an abelian category admitting enough injective objects.
Let $X=(X^{i},d^{i})$ be a bounded complex in $\mathcal{A}$ such that
the homology $H^{i}(X)$ has injective dimension at most
$n$ for all $i$.Then $X\in\langle \mathcal{E}\rangle_{n+1}  \subseteq D^{b}(\mathcal{A})$
for the subcategory $\mathcal{E} \subseteq \mathcal{A}$ of injective objects.
\end{lemma}

\subsection{ Radical layer lengths and torsion pairs}

We recall some notions from \cite{HLM2}.
Let $\mathcal{C}$ be a {\bf length-category}, that is, $\mathcal{C}$
is an abelian, skeletally small category and every object of $\mathcal{C}$ has a finite composition series.
We use $\End_{\mathbb{Z}}(\mathcal{C})$ to denote the category of all additive functors from
$\mathcal{C}$ to $\mathcal{C}$, and use $\rad$ to denote the Jacobson radical lying in $\End_{\mathbb{Z}}(\mathcal{C})$.
For any $\alpha\in\End_{\mathbb{Z}}(\mathcal{C})$, set the {\bf $\alpha$-radical functor} $F_{\alpha}:=\rad\circ \alpha$.

\begin{definition}{\rm (\cite[Definition 3.1]{HLM2})\label{def2.5}
For any $\alpha, \beta \in \End_{\mathbb{Z}}(\mathcal{C})$, we define
the {\bf $(\alpha,\beta)$-layer length} $\ell\ell_{\alpha}^{\beta}:\mathcal{C} \longrightarrow \mathbb{N}\cup \{\infty\}$ via
$\ell\ell_{\alpha}^{\beta}(M)=\inf\{ i \geqslant 0\mid \alpha \circ \beta^{i}(M)=0 \}$; and
the {\bf $\alpha$-radical layer length} $\ell\ell^{\alpha}:=\ell\ell_{\alpha}^{F_{\alpha}}$.}
\end{definition}

\begin{lemma}\label{lem2.6}
Let $\alpha,\beta \in\End_{\mathbb{Z}}(\mathcal{C}) $.
For any $M\in \mathcal{C}$, if $\ell\ell_{\alpha}^{\beta}(M)=n$, then $\ell\ell_{\alpha}^{\beta}(M)=\ell\ell_{\alpha}^{\beta}(\beta^{j}(M))+j$
for any $0 \leqslant j\leqslant n$; in particular, if $\ell\ell^{\alpha}(M)=n$, then $\ell\ell^{\alpha}(F_{\alpha}^{n}(M))=0$.
\end{lemma}

\begin{proof}
If $\ell\ell_{\alpha}^{\beta}(M)=n$, then
$n=\inf\{ i \geqslant 0\mid \alpha \beta^{i}(M)=0 \}$.
By Definition \ref{def2.5}, for any $0 \leqslant j\leqslant n$, we have
\begin{align*}
\ell\ell_{\alpha}^{\beta}(\beta^{j}(M))=\inf\{i\geqslant 0\,|\, \alpha \beta^{i+j}(M)=0\}=n-j,
\end{align*}
that is, $\ell\ell_{\alpha}^{\beta}(M)=\ell\ell_{\alpha}^{\beta}(\beta^{j}(M))+j$.
In particular, if $\ell\ell^{\alpha}(M)=n$, then putting $\beta=F_{\alpha}$ we have $\ell\ell^{\alpha}(F_{\alpha}^{n}(M))=\ell\ell^{\alpha}(M)-n=n-n=0$.
\end{proof}

Recall that a {\bf torsion pair} (or {\bf torsion theory}) for $\mathcal{C}$
is a pair of classes $(\mathcal{T},\mathcal{F})$ of objects in $\mathcal{C}$ satisfying the following conditions.
\begin{enumerate}
\item[(1)] $\Hom_{\mathcal{C}}(M,N)=0$ for any $M\in\mathcal{T}$ and $N\in\mathcal{F}$;
\item[(2)] an object $X \in \mathcal{C}$ is in $\mathcal{T}$ if $\Hom_{\mathcal{C}}(X,-)|_{\mathcal{F}}=0$;
\item[(3)] an object $Y\in\mathcal{C}$ is in $\mathcal{F}$ if $\Hom_{\mathcal{C}}(-,Y)|_{\mathcal{T}}=0$.
\end{enumerate}

For a subfunctor $\alpha$ of the identity functor $1_{\C}$, we write $q_{\alpha}:=1_{\mathcal{C}}/\alpha$.
Let $(\mathcal{T},\mathcal{F})$ be a torsion pair for $\mathcal{C}$.
Recall that the {\bf torsion radical} $t$ is a functor in $\End_{\mathbb{Z}}(\mathcal{C})$ such that
$$0 \longrightarrow  t(M)\longrightarrow M \longrightarrow  q_{t}(M)\longrightarrow 0$$
is a short exact sequence and $q_{t}(M)=M/t(M)\in \mathcal{F}$.

\section{Main results}

In this section, $\Lambda$ is an artin algebra. Then $\mod \Lambda$ is a length-category.
For a module $M$ in $\mod\Lambda$, we use $\rad M$, $\soc M$ and $\top M$ to denote the radical, socle and top of $M$ respectively.
For a subclass $\mathcal{W}$ of $\mod \Lambda$, we use $\add \mathcal{W}$ to denote the subcategory
of $\mod \Lambda$ consisting of direct summands of finite direct sums of modules in $\mathcal{W}$,
and if $\mathcal{W}=\{M\}$ for some $M\in \mod \Lambda$, we write $\add M:=\add \mathcal{W}$.

Let $\mathcal{S}$ be the set of all simple modules in $\mod \Lambda$, and let $\mathcal{V}$ be a subset of $\mathcal{S}$
and $\mathcal{V}'$ the set of all the others simple modules in $\mod \Lambda$, that is, $\mathcal{V}'=\mathcal{S}\backslash\mathcal{V}$.
We write $\mathfrak{F}\,(\mathcal{V}):=\{M\in\mod\Lambda\mid$ there exists a finite chain
$$0=M_0\subseteq M_1\subseteq \cdots\subseteq M_m=M$$ of submodules of $M$
such that each quotient $M_i / M_{i-1}$ is isomorphic to some module in $\mathcal{V}\}$.
By \cite[Lemma 5.7 and Proposition 5.9]{HLM2}, we have that
$(\mathcal{T}_{\mathcal{V}}, \mathfrak{F}(\mathcal{V}))$ is a torsion pair, where
$$\mathcal{T}_{\mathcal{V}}=\{M \in \mod \Lambda\mid\top M\in \add \mathcal{V}'\}.$$
We use $t_{\mathcal{V}}$ to denote the torsion radical of the torsion pair $(\mathcal{T}_{\mathcal{V}}, \mathfrak{F}(\mathcal{V}))$.
Then $t_{\mathcal{V}}(M)\in \mathcal{T}_{\mathcal{V}}$ and $q_{_{t_{\mathcal{V}}}}(M)\in\mathfrak{F}(\mathcal{V})$ for any
$M\in \mod \Lambda$. By \cite[Propositions 5.3 and 5.9(a)]{HLM2}, we have

\begin{proposition}\label{prop3.1}
\item[(1)] $\mathfrak{F}(\mathcal{V})=\{ M\in \mod \Lambda \;|\; t_{\mathcal{V}}(M)=0\}$;
\item[(2)] $\mathcal{T}_{\mathcal{V}}=\{ M\in \mod \Lambda \;|\; t_{\mathcal{V}}(M)= M\}$;
\item[(3)] $\top M\in \add \mathcal{V}'$ if and only if $t_{\mathcal{V}}(M)=M$.
\end{proposition}

As a consequence, we get the following

\begin{proposition}\label{prop3.2}
If $\mathcal{V}=\varnothing$, then $\ell\ell^{t_{\mathcal{V}}}(M)=\LL(M)$ for any $M\in\mod \Lambda$.
\end{proposition}

\begin{proof}
If $\mathcal{V}=\varnothing$, then the torsion pair $(\mathcal{T}_{\mathcal{V}}, \mathfrak{F}(\mathcal{V}))=(\mod \Lambda,0)$.
By Proposition \ref{prop3.1}(3), for any $M\in \mod \Lambda$ we have $t_{\mathcal{V}}(M)=M$ and $\ell\ell^{t_{\mathcal{V}}}(M)=\LL(M)$.
\end{proof}

\begin{lemma}\label{lem3.3}
\begin{enumerate}
\item[]
\item[(1)]
$\mathfrak{F}(\mathcal{V})$ is closed under extensions, submodules and quotient modules.
\item[(2)] The functor $t_{\V}$ preserves monomorphisms and epimorphisms.
\end{enumerate}
\end{lemma}

\begin{proof}
(1) It is \cite[Lemma 5.7]{HLM2}.

(2) By \cite[Lemma 3.6(a)]{HLM2}, we have that $t_{\V}$ preserves monomorphisms.
Since $\mathfrak{F}(\mathcal{V})$ is closed under quotient modules by (1),
we have that $t_{\mathcal{V}}$ preserves epimorphisms by \cite[Proposition 1.3]{B71}.
\end{proof}

We use $\mathbb{D}$ to denote the usual duality between $\mod \Lambda$ and $\mod \Lambda^{op}$.

\begin{proposition}\label{prop3.4}
Let $G$ be a generator and $E$ a cogenerator for $\mod \Lambda$. Then
$\ell\ell^{t_{\V}}(G)=\ell\ell^{t_{\V}}(E)$.
In particular, for any $M\in \mod \Lambda$, we have
$$\ell\ell^{t_{\V}}(M)\leqslant\ell\ell^{t_{\V}}(\Lambda)=\ell\ell^{t_{\V}}(\mathbb{D}(\Lambda)).$$
\end{proposition}

\begin{proof}
By Lemma \ref{lem3.3}(2) and \cite[Lemma 3.4(b)(c)]{HLM2}.
\end{proof}

The following lemma is essentially contained in \cite[Lemma 2.2.4]{Opp07U}.
A similar result also holds true for objects in the bounded derived category of a hereditary abelian category
(see \cite[1.6]{Kr} for details).

\begin{lemma}\label{lem3.5}
Let
\[\xymatrix{
X:\ \cdots   \ar[rr]^{d^{i-2}} & & X^{i-1}\ar[rr]^{d^{i-1}}    & & X^{i}\ar[rr]^{d^{i}}
& & X^{i+1}\ar[rr]^{d^{i+1}}   &&\cdots   & }\]
be a bounded complex in $\mod \Lambda$ with all $X^{i}$ seimisimple.
Then $X\cong \oplus_{i}H^{i}(X)[i]$ and $X\in \langle \Lambda/\rad\Lambda \rangle$
in $D^{b}(\mod \Lambda)$.
\end{lemma}

\begin{proof}
By assumption, there exist two integers $r$ and $t$ such that $X^{i}\in \add (\Lambda/\rad\Lambda)$, where $X^{i}=0$ for any
$i\notin [r,t]$, where $[r,t]$ is the integer interval with endpoints $r$ and $t$.
By \cite[Theorem 9.6]{AF92R}, the exact sequence
$$0 \longrightarrow \Ker d^{t-1} \longrightarrow X^{t-1} \longrightarrow  \Im d^{t-1}\longrightarrow 0$$
splits. So the following complex
\[\xymatrix{
0\ar[r]&X^{r}   \ar[r]^{d^{r}} & X^{r+1}\ar[r]^{d^{r+1}}    & X^{r+2}\ar[r]^{d^{r+2}}    &\cdots \ar[r]^{d^{t-2}}
& X^{t-1}\ar[r]^{d^{t-1}}   & X^{t}\ar[r]   &0  &}\]
is the direct sum of the following two complexes
\[\xymatrix{
0\ar[r]&X^{r}   \ar[r]^{d^{r}} & X^{r+1}\ar[r]^{d^{r+1}}    & X^{r+2}\ar[r]^{d^{r+2}}    &\cdots \ar[r]^(0.4){d^{t-2}}
& \Ker d^{t-1} \ar[r]  & 0 \ar[r]  &0  &}\]
and
\[\xymatrix{
0\ar[r]&0   \ar[r]& 0\ar[r]    & 0\ar[r]& \cdots \ar[r]
&\Im d^{t-1}\ar[r]   & X^{t}\ar[r]   &0. &(*)}\]
Note that the complex $(*)$ is isomorphic to the stalk complex $H^{t} (X)[t]$ in $D^{b}(\mod\Lambda)$.
By induction, we have $X\cong \oplus _{i=r}^{t}H^{i}(X)[i]$ in $D^{b}(\mod\Lambda)$.
\end{proof}

\subsection{An upper bound for $\dim D^{b}(\mod \Lambda)$}

We use $\mathcal{S}^{<\infty}$ to denote the set of the
simple modules in $\mod \Lambda$ with finite projective dimension, and use $\mathcal{S}^{\infty}$ to denote the set of the
simple modules in $\mod \Lambda$ with infinite projective dimension. Thus $\mathcal{S}^{<\infty}\cup\mathcal{S}^{\infty}=\mathcal{S}$.
For a subset $\mathcal{V}$ of $\mathcal{S}$, it is easy to see that $\pd\mathfrak{F}(\mathcal{V})\leqslant\pd\mathcal{V}$
and $\id\mathfrak{F}(\mathcal{V})\leqslant\id\mathcal{V}$. We will use this observation in the sequel freely.

\begin{lemma}\label{lem3.6}
Let $\mathcal{V}$ be a subset of $\mathcal{S}^{<\infty}$ and $\pd \mathcal{V}=a$. Then the following complex
\[\xymatrix{
X:\ \cdots   \ar[rr]^{d^{i-2}} & & X^{i-1}\ar[rr]^{d^{i-1}}    & & X^{i}\ar[rr]^{d^{i}}
& & X^{i+1}\ar[rr]^{d^{i+1}}   &&\cdots   &}\]
with all $X^{i}$ in $\mod \Lambda$ induces a complex
\[\xymatrix{
q_{t_{\mathcal{V}}}(X):\ \cdots   \ar[rr]^{q_{t_{\mathcal{V}}}(d^{i-2})} && q_{t_{\mathcal{V}}}(X^{i-1})\ar[rr]^{q_{t_{\mathcal{V}}}(d^{i-1})}
&& q_{t_{\mathcal{V}}}(X^{i}) \ar[rr]^{q_{t_{\mathcal{V}}}(d^{i})}  && q_{t_{\mathcal{V}}}(X^{i+1}) \ar[rr]^{q_{t_{\mathcal{V}}}(d^{i+1})}
& &\cdots   &}\]
such that $\pd H^{i}(q_{t_{\mathcal{V}}}(X))\leqslant a $ for all $i$.
\end{lemma}

\begin{proof}
Since $q_{t_{\mathcal{V}}}$ is a covariant functor, we can obtain the complex $q_{t_{\mathcal{V}}}(X)$. For any $i$,
since $q_{t_{\mathcal{V}}}(X^{i})\in \mathcal{ \mathfrak{F}(\mathcal{V})}$, it follows from Lemma \ref{lem3.3}(1) that all
$\Ker q_{t_{\mathcal{V}}}(d^{i})$, $\Im q_{t_{\mathcal{V}}}(d^{i-1})$ and $H^{i}(q_{t_{\mathcal{V}}}(X))$ are in $\mathcal{\mathfrak{F}(\mathcal{V})}$.
Thus we have $\pd H^{i}(q_{t_{\mathcal{V}}}(X))\leqslant a$.
\end{proof}

\begin{lemma}\label{lem3.7}
Let $\mathcal{V}$ be a subset of $\mathcal{S}^{<\infty}$ and $\pd \mathcal{V}=a$.
For a bounded complex $X=(X^i,d^i)$ in $\mod \Lambda$, if $\ell\ell^{t_{\mathcal{V}}}(\Lambda)=n$,
then $F^{n}_{t_{\mathcal{V}}}(X)\in  \langle  \Lambda\rangle_{a+1}$.
\end{lemma}

\begin{proof}
By Proposition \ref{prop3.4}, we have
$\ell\ell^{t_{\mathcal{V}}}(X^{i})\leqslant \ell\ell^{t_{\mathcal{V}}}(\Lambda)=n$ for all $i$.
Then by Lemma \ref{lem2.6} and Proposition \ref{prop3.1}(1),
we have $\ell\ell^{t_{\mathcal{V}}}(F^{n}_{t_{\mathcal{V}}}(X^{i}))=0$ and
$F^{n}_{t_{\mathcal{V}}}(X^{i})\in \mathfrak{F}(\mathcal{V})$, which implies $H^{i}(F^{n}_{t_{\mathcal{V}}}(X))\in\mathfrak{F}(\mathcal{V})$
by Lemma \ref{lem3.3}(1), and hence $\pd H^{i}(F^{n}_{t_{\mathcal{V}}}(X))\leqslant a$ for all $i$.
It follows from Lemma \ref{lem2.3} that $F^{n}_{t_{\mathcal{V}}}(X)\in \langle \Lambda \rangle_{a+1}$.
\end{proof}

We now are in a position to prove the following

\begin{theorem}\label{thm3.8}
Let $\mathcal{V}$ be a subset of $\mathcal{S}^{<\infty}$ and $\pd \mathcal{V}=a$.
If $\ell\ell^{t_{\mathcal{V}}}(\Lambda)=n$, then $$\dim D^{b}(\mod \Lambda) \leqslant (a+2)(n+1)-2.$$
\end{theorem}

\begin{proof}
If $\mathcal{V}=\varnothing$, then $\ell\ell^{t_{\mathcal{V}}}(\Lambda)=\LL(\Lambda)$
by Proposition \ref{prop3.2}. Now the assertion follows from Theorem \ref{thm1.1}(1).

If $n=0$, that is, $t_{\mathcal{V}}(\Lambda)=0$, then $\Lambda\in \mathfrak{F}(\mathcal{V})$ by Proposition \ref{prop3.1}(1).
Since $\mathcal{V}$ contains every simple module by the definition of $\mathfrak{F}(\mathcal{V})$ and since the composition
series of $\Lambda$ does, we have $\mathcal{V}=\mathcal{S}$
and $\gldim \Lambda=a$. It follows from Theorem \ref{thm1.1}(2) that $\dim D^{b}(\mod \Lambda)\leqslant a$.

Let $X\in D^{b}(\mod \Lambda)$ and $n\geqslant1$.
Since both $q_{t_{\mathcal{V}}}$ and $t_{\mathcal{V}}$ are covariant functors, we have that
$$0\longrightarrow t_{\mathcal{V}}(X)\longrightarrow X \longrightarrow q_{t_{\mathcal{V}}}(X)\longrightarrow 0$$
is a short exact sequence of complexes.
For any $Y\in D^{b}(\mod \Lambda)$, we have the following short exact sequence of complexes
$$0\longrightarrow \rad Y\longrightarrow Y \longrightarrow \top Y\longrightarrow 0.$$
Now by letting $Y=t_{\mathcal{V}}(X)$, we have
\begin{align*}
\langle X\rangle&\subseteq\langle t_{\mathcal{V}}(X) \rangle \diamond \langle q_{t_{\mathcal{V}}}(X)\rangle\\
&\subseteq \langle t_{\mathcal{V}}(X) \rangle \diamond \langle \Lambda\rangle_{a+1}\ \ \text{(by Lemmas \ref{lem3.6} and \ref{lem2.3})}\\
&\subseteq\langle \rad t_{\mathcal{V}}(X)\rangle \diamond  \langle \top t_{\mathcal{V}}(X)\rangle\diamond \langle \Lambda\rangle_{a+1}\\
&=\langle F_{t_{\mathcal{V}}}(X)\rangle \diamond  \langle \top t_{\mathcal{V}}(X)\rangle\diamond \langle \Lambda\rangle_{a+1}\\
&\subseteq\langle F_{t_{\mathcal{V}}}(X)\rangle \diamond  \langle  \Lambda / \rad\Lambda\rangle\diamond \langle \Lambda\rangle_{a+1}\ \ \text{(by Lemma \ref{lem3.5})}\\
&\subseteq\langle F_{t_{\mathcal{V}}}(X)\rangle \diamond  \langle  \Lambda\oplus (\Lambda / \rad\Lambda)\rangle_{a+2}.\ \ \text{(by Lemma \ref{lem2.2})}
\end{align*}
By replacing $X$ with $F^i_{t_{\mathcal{V}}}(X)$ for any $1\leqslant i\leqslant n-1$, we get
$$\langle X\rangle \subseteq\langle F^{n}_{t_{\mathcal{V}}}(X)\rangle \diamond \langle \Lambda\oplus (\Lambda/\rad\Lambda)\rangle_{n(a+2)}.$$
By Lemma \ref{lem3.7}, we have
$F^{n}_{t_{\mathcal{V}}}(X)\in  \langle  \Lambda\rangle_{a+1}$. Thus
$$ \langle X\rangle \subseteq   \langle  \Lambda\oplus(\Lambda/\rad\Lambda)\rangle_{(n+1)(a+2)-1}.$$
It follows that $D^{b}(\mod \Lambda)=\langle  \Lambda\oplus (\Lambda / \rad\Lambda)\rangle_{(a+2)(n+1)-1}$ and
$$\dim D^{b}(\mod \Lambda) \leqslant (a+2)(n+1)-2.$$
\end{proof}

We use $\S_{inj}^{<\infty}$ to denote the set of the simple modules in $\mod \Lambda$ with finite injective dimension.
The following two lemmas are dual to Lemmas \ref{lem3.6} and \ref{lem3.7} respectively, we omit their proofs.

\begin{lemma}\label{lem3.9}
Let $\mathcal{V}$ be a subset of $\mathcal{S}_{inj}^{<\infty}$ and $\id \mathcal{V}=c$.
Then the following complex
\[\xymatrix{
X:\ \cdots   \ar[rr]^{d^{i-2}} & & X^{i-1}\ar[rr]^{d^{i-1}}    & & X^{i}\ar[rr]^{d^{i}}
& & X^{i+1}\ar[rr]^{d^{i+1}}   &&\cdots   &}\]
with all $X^{i}$ in $\mod \Lambda$ induces a complex
\[\xymatrix{
q_{t_{\mathcal{V}}}(X):\ \cdots   \ar[rr]^{q_{t_{\mathcal{V}}}(d^{i-2})} && q_{t_{\mathcal{V}}}(X^{i-1})\ar[rr]^{q_{t_{\mathcal{V}}}(d^{i-1})}
&& q_{t_{\mathcal{V}}}(X^{i}) \ar[rr]^{q_{t_{\mathcal{V}}}(d^{i})}  && q_{t_{\mathcal{V}}}(X^{i+1}) \ar[rr]^{q_{t_{\mathcal{V}}}(d^{i+1})}
& &\cdots   &}\]
such that $\id H^{i}(q_{t_{\mathcal{V}}}(X))\leqslant c $ for all $i$.
\end{lemma}

\begin{lemma}\label{lem3.10}
Let $\mathcal{V}$ be a subset of $\mathcal{S}_{inj}^{<\infty}$ and $\id \mathcal{V}=c$.
For a bounded complex $X=(X^i,d^i)$ in $\mod \Lambda$, if $\ell\ell^{t_{\mathcal{V}}}(\mathbb{D}(\Lambda))=n$,
then $F^{n}_{t_{\mathcal{V}}}(X)\in  \langle \mathbb{D}(\Lambda)\rangle_{c+1}$.
\end{lemma}

The following result is dual to Theorem \ref{thm3.8}.

\begin{theorem}\label{thm3.11}
Let $\mathcal{V}$ be a subset of $\mathcal{S}_{inj}^{<\infty}$ and $\id \mathcal{V}=c$.
If $\ell\ell^{t_{\mathcal{V}}}(\mathbb{D}(\Lambda))=n$, then $$\dim D^{b}(\mod \Lambda ) \leqslant (c+2)(n+1)-2.$$
\end{theorem}

\begin{proof}
Though the proof is similar to that of Theorem \ref{thm3.8}, we still give it here for the readers' convenience.

If $\mathcal{V}=\varnothing$, then $\ell\ell^{t_{\mathcal{V}}}(\mathbb{D}(\Lambda))=\LL(\mathbb{D}(\Lambda))=\LL(\Lambda)$
by Proposition \ref{prop3.2}. Now the assertion follows from Theorem \ref{thm1.1}(1).

If $n=0$, that is, $t_{\mathcal{V}}(\mathbb{D}(\Lambda))=0$, then $\mathbb{D}(\Lambda)\in \mathfrak{F}(\mathcal{V})$ by Proposition \ref{prop3.1}(1).
Since $\mathcal{V}$ contains every simple module by the definition of $\mathfrak{F}(\mathcal{V})$ and since the composition
series of $\mathbb{D}(\Lambda)$ does, we have $\mathcal{V}=\mathcal{S}$
and $\gldim \Lambda=c$. It follows from Theorem \ref{thm1.1}(2) that $\dim D^{b}(\mod \Lambda)\leqslant c$.

Let $X,Y\in D^{b}(\mod \Lambda)$ and $n\geqslant1$. Just like the argument in Theorem \ref{thm3.8}, we have the
following two short exact sequence of complexes
$$0\longrightarrow t_{\mathcal{V}}(X)\longrightarrow X \longrightarrow q_{t_{\mathcal{V}}}(X)\longrightarrow 0,$$
$$0\longrightarrow \rad Y\longrightarrow Y \longrightarrow \top Y\longrightarrow 0.$$
Now by letting $Y=t_{\mathcal{V}}(X)$, we have
\begin{align*}
\langle X\rangle&\subseteq\langle t_{\mathcal{V}}(X) \rangle \diamond \langle q_{t_{\mathcal{V}}}(X)\rangle\\
&\subseteq \langle t_{\mathcal{V}}(X) \rangle \diamond \langle \mathbb{D}(\Lambda)\rangle_{c+1}\ \ \text{(by Lemmas \ref{lem3.9} and \ref{lem2.4})}\\
&\subseteq\langle \rad t_{\mathcal{V}}(X)\rangle \diamond  \langle \top t_{\mathcal{V}}(X)\rangle\diamond \langle \mathbb{D}(\Lambda)\rangle_{c+1}\\
&=\langle F_{t_{\mathcal{V}}}(X)\rangle \diamond  \langle \top t_{\mathcal{V}}(X)\rangle\diamond \langle \mathbb{D}(\Lambda)\rangle_{c+1}\\
&\subseteq\langle F_{t_{\mathcal{V}}}(X)\rangle \diamond  \langle  \Lambda / \rad\Lambda\rangle\diamond \langle \mathbb{D}(\Lambda)\rangle_{c+1}\ \ \text{(by Lemma \ref{lem3.5})}\\
&\subseteq\langle F_{t_{\mathcal{V}}}(X)\rangle \diamond  \langle  \mathbb{D}(\Lambda)\oplus (\Lambda / \rad\Lambda)\rangle_{c+2}.\ \ \text{(by Lemma \ref{lem2.2})}
\end{align*}
By replacing $X$ with $F^i_{t_{\mathcal{V}}}(X)$ for any $1\leqslant i\leqslant n-1$, we get
$$\langle X\rangle \subseteq\langle F^{n}_{t_{\mathcal{V}}}(X)\rangle \diamond \langle \mathbb{D}(\Lambda)\oplus (\Lambda/\rad\Lambda)\rangle_{n(c+2)}.$$
By Lemma \ref{lem3.10}, we have
$F^{n}_{t_{\mathcal{V}}}(X)\in  \langle  \mathbb{D}(\Lambda)\rangle_{c+1}$. Thus
$$ \langle X\rangle \subseteq   \langle  \mathbb{D}(\Lambda)\oplus(\Lambda/\rad\Lambda)\rangle_{(n+1)(c+2)-1}.$$
It follows that $D^{b}(\mod \Lambda)=\langle  \mathbb{D}(\Lambda)\oplus (\Lambda / \rad\Lambda)\rangle_{(c+2)(n+1)-1}$ and
$$\dim D^{b}(\mod \Lambda) \leqslant (c+2)(n+1)-2.$$
\end{proof}

Combining Theorems \ref{thm3.8} and \ref{thm3.11}, we get the following

\begin{theorem}\label{thm3.12}
Let $\mathcal{V}$ be a subset of $\mathcal{S}$ and $\min\{\pd\mathcal{V}, \id\mathcal{V}\}=d$.
If $\ell\ell^{t_{\mathcal{V}}}(\Lambda)=n$, then $$\dim D^{b}(\mod \Lambda) \leqslant (d+2)(n+1)-2.$$
\end{theorem}

\begin{proof}
The case for $d=\infty$ is trivial. Since $\ell\ell^{t_{\mathcal{V}}}(\Lambda)=\ell\ell^{t_{\mathcal{V}}}(\mathbb{D}(\Lambda))$
by Proposition \ref{prop3.4}, the case for $d<\infty$ follows from Theorems \ref{thm3.8} and \ref{thm3.11}.
\end{proof}

\subsection{An upper bound for $\dim D_{sg}^{b}(\mod \Lambda)$}

Recall that the {\bf singularity category} $D^{b}_{sg}(\mod \Lambda)$ of $\mod \Lambda$ is defined as $D^{b}(\mod \Lambda)/ K^{b}(\proj \Lambda)$,
where $K^{b}(\proj \Lambda)$ is the bounded homotopy category of the subcategory $\proj \Lambda$ of $\mod \Lambda$ consisting of projective modules.
For any $M\in\mod \Lambda$ and $m\geqslant 1$, we use $\Omega^{m}(M)$ to denote the $m$-th syzygy of $M$; in particular, $\Omega^{0}(M)=M$.

\begin{lemma}\label{lem3.13}
\begin{enumerate}
\item[]
\item[(1)] $\ell\ell^{t_{\mathcal{S}^{<\infty}}}(\Lambda)=0$ if and only if $\gldim \Lambda <\infty$;
\item[(2)] $\ell\ell^{t_{\mathcal{S}^{<\infty}}}(\Lambda)\neq 1$.
\end{enumerate}
\end{lemma}

\begin{proof}
(1) If $\ell\ell^{t_{\mathcal{S}^{<\infty}}}(\Lambda)=0$,
then $t_{\mathcal{S}^{<\infty}}(\Lambda)=0$. So $\Lambda\in \mathfrak{F}(\mathcal{S}^{<\infty})$ by Proposition \ref{prop3.1}(1),
which implies $\mathcal{S}^{<\infty}=\mathcal{S}$. Thus $\gldim \Lambda=\pd\mathcal{S}=\pd\mathcal{S}^{<\infty}<\infty$.
Conversely, if $\gldim \Lambda <\infty$, then $\mathcal{S}^{<\infty}=\mathcal{S}$ and the torsion pair
$(\mathcal{T}_{\mathcal{S}^{<\infty}}, \mathfrak{F}(\mathcal{S}^{<\infty}))=(\mathcal{T}_{\mathcal{S}}, \mathfrak{F}(\mathcal{S}))
=(0,\mod \Lambda)$. By Proposition \ref{prop3.1}(2), for any $M\in \mod \Lambda$ we have $t_{\mathcal{S}^{<\infty}}(M)=0$ and
$\ell\ell^{t_{\mathcal{S}^{<\infty}}}(\Lambda)=0$.

(2) Suppose $\ell\ell^{t_{\mathcal{S}^{<\infty}}}(\Lambda)=1$. Then by (1), we have $\gldim \Lambda=\infty$
and there exists a simple module $S$ in $\mod \Lambda$
such that $\pd S=\infty$. Consider the following exact sequence
$$0 \longrightarrow \Omega^1(S) \longrightarrow P \longrightarrow S \longrightarrow 0,$$
in $\mod \Lambda$ with $P$ the projective cover of $S$.
Because $\top S=S\in \add \mathcal{S}^{\infty}$, we have $t_{\mathcal{S}^{<\infty}}(S)=S$ by
Proposition \ref{prop3.1}(3). 
It follows from \cite[Lemma 6.3]{HLM2} that
$$\ell\ell^{t_{\mathcal{S}^{<\infty}}}(\Omega^1(S))=\ell\ell^{t_{\mathcal{S}^{<\infty}}}(\Omega^1(t_{\mathcal{S}^{<\infty}}(S)))
\leqslant \ell\ell^{t_{\mathcal{S}^{<\infty}}}(\Lambda)-1=0,$$
that is, $\ell\ell^{t_{\mathcal{S}^{<\infty}}}(\Omega^1(S))=0$, and $\Omega^1(S)\in \mathfrak{F}(\mathcal{S}^{<\infty})$,
which induces $\pd \Omega^1(S) <\infty$, a contradiction.
\end{proof}

In the following result, we give an upper bound for $\dim D_{sg}^{b}(\mod \Lambda)$.

\begin{theorem}\label{thm3.14}
Let $\mathcal{V}$ be a subset of $\mathcal{S}^{<\infty}$ with $\ell\ell^{t_{\mathcal{V}}}(\Lambda)=n$. Then we have
\begin{equation*}
\dim D_{sg}^{b}(\mod \Lambda)
\leqslant \max\{0, n-2\}.
\end{equation*}
\end{theorem}

\begin{proof}
If $\mathcal{V}=\varnothing$, then $\ell\ell^{t_{\mathcal{V}}}(\Lambda)=\LL(\Lambda)$
by Proposition \ref{prop3.2}. Now the assertion follows from Theorem \ref{thm1.1}(3).

Now suppose $\mathcal{V}\neq\varnothing$.
If $n\leqslant 1$, then $\ell\ell^{t_{\mathcal{S}^{<\infty}}}(\Lambda)\leqslant 1$ by \cite[Proposition 5.10]{HLM2}.
So $\ell\ell^{t_{\mathcal{S}^{<\infty}}}(\Lambda)=0$
and $\gldim \Lambda<\infty$ by Lemma \ref{lem3.13}, which implies $\dim D_{sg}^{b}(\mod \Lambda)=0$.

Let $n\geqslant 2$ and set $a:=\pd\mathcal{V}$.
From \cite[Lemma 2.4(2)(a)]{DT15U}, we know that every object in $ D_{sg}^{b}(\mod \Lambda)$
is isomorphic to a stalk complex for some module. Let $X\in \mod \Lambda$.
If $\ell\ell^{t_{\mathcal{V}}}(X)=0$, then $\pd X<\infty$ and $X=0$ in $D_{sg}^{b}(\mod \Lambda)$.
If $\ell\ell^{t_{\mathcal{V}}}(X)>0$, then by \cite[Lemma 6.3]{HLM2}, we have
$\ell\ell^{t_{\mathcal{V}}}(\Omega^{1} (t_{\mathcal{V}}(X)))\leqslant \ell\ell^{t_{\mathcal{V}}}(\Lambda)-1=n-1$. By Lemma \ref{lem2.6}, we have
$\ell\ell^{t_{\mathcal{V}}}(F_{t_{\mathcal{V}}}^{n-1}(\Omega^{1} (t_{\mathcal{V}}(X))))=0$.
By Proposition \ref{prop3.1}(1), we have $F_{t_{\mathcal{V}}}^{n-1}(\Omega^{1} (t_{\mathcal{V}}(X)))\in\mathfrak{F}(\mathcal{V})$
and $\pd F_{t_{\mathcal{V}}}^{n-1}(\Omega^{1} (t_{\mathcal{V}}(X)))\leqslant a$.

For any $Y\in \mod \Lambda$, we have the following two exact sequences
$$0 \longrightarrow t_{\mathcal{V}}(Y)\longrightarrow Y\longrightarrow q_{t_{\mathcal{V}}}(Y)\longrightarrow 0, $$
$$0 \longrightarrow F_{t_{\mathcal{V}}}(Y)\longrightarrow t_{\mathcal{V}}(Y)\longrightarrow \top t_{\mathcal{V}}(Y)\longrightarrow 0.$$
Since $q_{t_{\mathcal{V}}}(Y)\in\mathfrak{F}(\mathcal{V})$, we have $\pd q_{t_{\mathcal{V}}}(Y)\leqslant a$.
By the horseshoe lemma, we have
$$\Omega^{a+1}(Y)\cong \Omega^{a+1}(t_{\mathcal{V}}(Y)),$$
$$0 \rightarrow \Omega^{a+1}(F_{t_{\mathcal{V}}}(Y))\rightarrow \Omega^{a+1}(t_{\mathcal{V}}(Y))\oplus P_{1}\rightarrow
\Omega^{a+1}(\top t_{\mathcal{V}}(Y)) \rightarrow 0,$$
where $P_{1}$ is projective in $\mod \Lambda$.
Thus we have
\begin{align*}
\langle\Omega^{a+1}(Y)\rangle=\langle \Omega^{a+1}(t_{\mathcal{V}}(Y))\rangle
&\subseteq \langle \Omega^{a+1}(F_{t_{\mathcal{V}}}(Y))\rangle \diamond \langle \Omega^{a+1}(\top t_{\mathcal{V}}(Y))\rangle\\
&\subseteq \langle \Omega^{a+1}(F_{t_{\mathcal{V}}}(Y))\rangle \diamond \langle \Omega^{a+1}(\Lambda/\rad \Lambda)\rangle.
\end{align*}
By replacing $Y$ with $F^i_{t_{\mathcal{V}}}(Y)$ for any $1\leqslant i\leqslant n-2$, we get
$$
\langle\Omega^{a+1}(Y)\rangle\subseteq \langle \Omega^{a+1}(F^{n-1}_{t_{\mathcal{V}}}(Y))\rangle \diamond \langle \Omega^{a+1}(\Lambda/\rad \Lambda)\rangle_{n-1}.
$$
Let $Y=\Omega^{1} (t_{\mathcal{V}}(X))$.
Since $\pd F_{t_{\mathcal{V}}}^{n-1}(\Omega^{1} (t_{\mathcal{V}}(X)))\leqslant a$, we have
$$\Omega^{a+1}(F_{t_{\mathcal{V}}}^{n-1}(\Omega^{1} (t_{\mathcal{V}}(X))))=0,$$
and so $$
\langle\Omega^{a+2}(t_{\mathcal{V}}(X))\rangle\subseteq \langle \Omega^{a+1}(\Lambda/\rad \Lambda)\rangle_{n-1}.
$$
By \cite[Lemma 2.4(2)(b)]{DT15U}, we have $X\cong \Omega^{a+2}(X)[a+2]$ in $D_{sg}^{b}(\mod \Lambda)$. Thus
$$X\cong \Omega^{a+2}(X)[a+2]\cong\Omega^{a+2}(t_{\mathcal{V}}(X))[a+2]\in \langle \Omega^{a+1}(\Lambda/\rad \Lambda)\rangle_{n-1}.$$
It follows that $D_{sg}^{b}(\mod \Lambda)=\langle \Omega^{a+1}(\Lambda/\rad\Lambda)\rangle_{n-1}$ and
$\dim D_{sg}^{b}(\mod \Lambda) \leqslant n-2$.
\end{proof}

The following corollary is an immediate consequence of Theorem \ref{thm3.14}. It is trivial that
$\ell\ell^{t_{\mathcal{S}^{<\infty}}}(\Lambda)\leqslant\LL (\Lambda)$, so this corollary improves Theorem \ref{thm1.1}(3).

\begin{corollary}\label{cor3.15}
If $\ell\ell^{t_{\mathcal{S}^{<\infty}}}(\Lambda)=n$, then we have
\begin{equation*}
\dim D_{sg}^{b}(\mod \Lambda)
\leqslant \max\{0, n-2\}.
\end{equation*}
\end{corollary}

Now we explain why Theorem \ref{thm1.1} is a special case of our results.

\begin{remark}\label{rem3.16}
\begin{enumerate}
\item[]
\end{enumerate}
{\rm (1) If $\mathcal{V}=\varnothing$, then $\ell\ell^{t_{\mathcal{V}}}(\Lambda)=\LL(\Lambda)$ by Proposition \ref{prop3.2}.
Since $c =\min\{\pd \mathcal{V}, \id \mathcal{V}\}=-1$, by Theorem \ref{thm3.12} we have
$$\dim D^{b}(\mod \Lambda) \leqslant (c+2)(n+1)-2=(-1+2)(\LL(\Lambda)+1)-2=\LL(\Lambda)-1.$$
This is Theorem \ref{thm1.1}(1).

By Theorem \ref{thm3.14}, we have
\begin{equation*}
\dim D_{sg}^{b}(\mod \Lambda)\leqslant \max\{0, \LL(\Lambda)-2\}.
\end{equation*}
This is Theorem \ref{thm1.1}(3).

(2) If $\mathcal{V}=\mathcal{S}^{<\infty}=\mathcal{S}$, then the torsion pair $(\mathcal{T}_{\mathcal{V}}, \mathfrak{F}(\mathcal{V}))=(0,\mod \Lambda)$.
By Proposition \ref{prop3.1}(2), for any $M\in \mod \Lambda$ we have $t_{\mathcal{V}}(M)=0$ and
$\ell\ell^{t_{\mathcal{V}}}(\Lambda)=0$. Because $c=\min\{\pd \mathcal{V}, \id \mathcal{V}\}=\gldim \Lambda<\infty$,
by Theorem \ref{thm3.12} we have
$$\dim D^{b}(\mod \Lambda) \leqslant (c+2)(\ell\ell^{t_{\mathcal{V}}}(\Lambda)+1)-2=(\gldim \Lambda+2)(0+1)-2=\gldim \Lambda.$$
This is Theorem \ref{thm1.1}(2). In addition, since $\gldim \Lambda<\infty$, we have $\dim D^{b}_{sg}(\mod \Lambda)=0$.}
\end{remark}

\section{Examples}

By choosing some suitable $\mathcal{V}$ and applying Theorems \ref{thm3.12} and \ref{thm3.14},
we may obtain more precise upper bounds for $\dim D^{b}(\mod \Lambda)$ and $\dim D^{b}_{sg}(\mod \Lambda)$
than that in Theorem \ref{thm1.1}. We give two examples to illustrate this.
The global dimension of the algebra in the first example is infinite and
that in the second one is finite.

\begin{example}\label{exa4.1}
{\rm Consider the bound quiver algebra $\Lambda=kQ/I$, where $k$ is an algebraically closed field and $Q$
is given by
$$
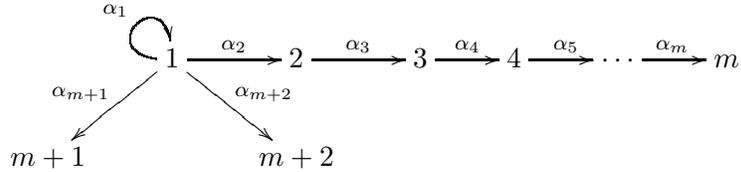
\xymatrix{
&1 \ar@(l,u)^{\alpha_{1}}\ar[r]^{\alpha_{2}}  \ar[ld]_{\alpha_{m+1}}\ar[rd]^{\alpha_{m+2}}
&2\ar[r]^{\alpha_{3}}&{3}\ar[r]^{\alpha_{4}}  &{4}\ar[r]^{\alpha_{5}}&\cdots\ar[r]^{\alpha_{m}}&m\\
m+1&&m+2&&&&
}$$
and $I$ is generated by
$\{\alpha_{1}^{2},\alpha_{1}\alpha_{m+1},\alpha_{1}\alpha_{m+2},\alpha_{1}\alpha_{2},
\alpha_{2}\alpha_{3}\cdots\alpha_{m}\}$ with $m\geq 10$.
Then the indecomposable projective $\Lambda$-modules are
$$
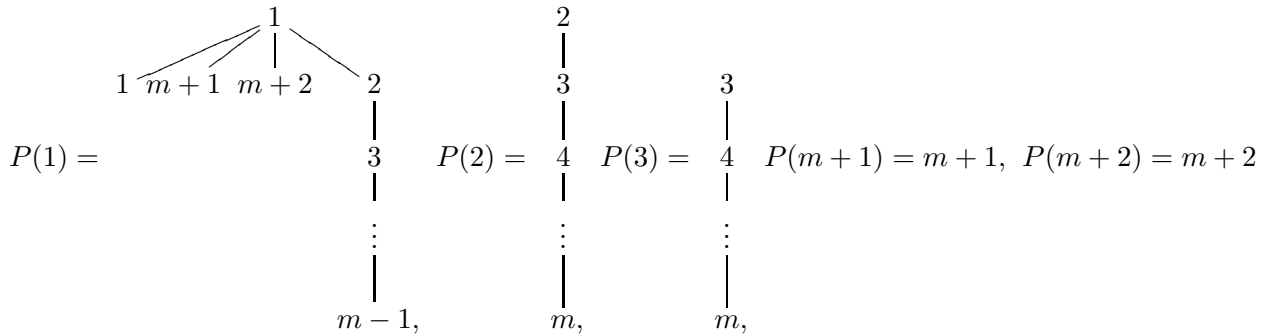
\xymatrix@-1.0pc@C=0.1pt
{ &  &  &1\edge[lld]\edge[ld]\edge[d]\edge[dr]
&  && 2\edge[d] &&&& & &&&&  &&&&  &\\
&1& m+1&m+2&2\edge[d] && 3\edge[d] && 3\edge[d] &&&& &&&& &\\
P(1)=&  &  &    &3\edge[d] &P(2)=&4\edge[d]  &P(3)=&4\edge[d] &P(m+1)=m+1,&P(m+2)=m+2&\\
&  &  &  &\vdots\edge[d]&&\vdots\edge[d]&&\vdots\edge[d]&& &&&& &\\
&  &  &  & \;m-1, &&\;m, && \;m,  &&&& &&&\\
&  &  &  &  &&&&  & &&&& & &&&& &&&\\
}$$
and $P(i+1)=\rad P(i)$ for any $2 \leqslant i\leqslant m-1$; and
the indecomposable injective $\Lambda$-modules are
$$ 
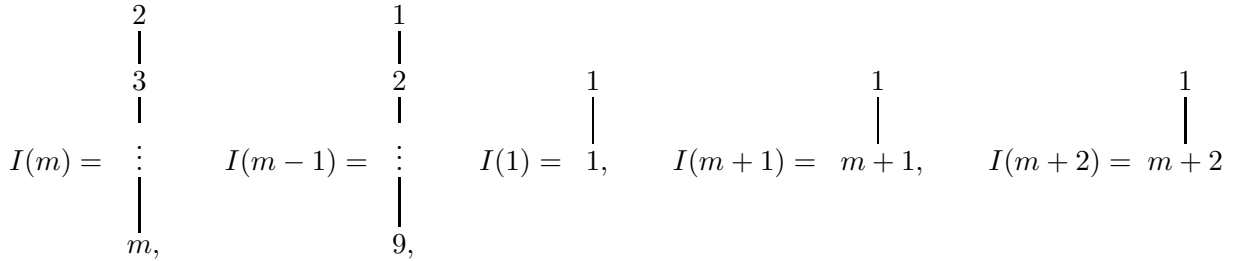
\xymatrix@-1.0pc@C=0.1pt
{& & 2\edge[d] &&&& & 1\edge[d] &&&&& &&&&  & &&&& &  \\
& &  3\edge[d] &&&& & 2\edge[d] &&&& & 1\edge[d]&&&&  &1\edge[d]&&&& &1\edge[d]\\
&  I(m)=&\vdots\edge[d]  &&&&I(m-1)= & \vdots\edge[d]  &&&& I(1)=&\;1, &&&&I(m+1)=&\;m+1,&&&&I(m+2)=&m+2 \\
&  &  \;m,  & &&&& \;9,  & &&&&  &   &&&&  &  &&&&   &\\
&  &   &   &&&&    &   &&&&   &   &&&&   &   &&&&  &\\}$$
and $I(i)=I(i+1)/\soc I(i+1)$ for any $2 \leqslant i\leqslant m-2$.

We have
\begin{equation*}
\pd S(i)=
\begin{cases}
\infty, &\text{if}\;\;i=1;\\
1,&\text{if} \;\;2 \leqslant  i\leqslant m-1;\\
0,&\text{if}\;\; m \leqslant  i\leqslant m+2.
\end{cases}
\end{equation*}
So $\mathcal{S}^{\infty}=\{S(1)\}$ and $\mathcal{S}^{<\infty}=\{ S(i)\mid 2\leqslant i\leqslant m+2\}$.
We also have \begin{equation*}
\id S(i)=
\begin{cases}
\infty, &\text{if}\;\;i=1,2,m,m+1,m+2;\\
1,&\text{if} \;\;3\leqslant  i\leqslant m-1.
\end{cases}
\end{equation*}
Let $\mathcal{V}:=\{S(i)\mid 3\leqslant i \leqslant m-1\}\subseteq\mathcal{S}^{<\infty}$. Then
$$a:=\pd\mathcal{S}=1,\;c:=\id\mathcal{S}=1\; {\rm and}\; d:=\min\{ a, c\}=1.$$
Let $\mathcal{V}'$ be all the others simple modules in $\mod \Lambda$, that is,
$\mathcal{V}'=\{ S(1),S(2),S(m),S(m+1),S(m+2)\}$.
By \cite[Lemma 3.4(a)]{HLM2} and $\Lambda=\oplus_{i=1}^{m+2}P(i)$, we have
$$\ell\ell^{t_{\mathcal{V}}}(\Lambda)=\max\{\ell\ell^{t_{\mathcal{V}}}(P(i)) \;|\; 1 \leqslant i  \leqslant m+2\}.$$

In order to compute $\ell\ell^{t_{\mathcal{V}}}(P(1))$, we need to find the least non-negative integer $i$
such that $t_{\mathcal{V}}F_{t_{\mathcal{V}}}^{i}(P(1))=0$.
Since $\top P(1)=S(1)\in \add \mathcal{V}'$, we have $t_{\mathcal{V}}(P(1))=P(1)$ by Proposition \ref{prop3.1}(3).
Thus
$$F_{t_{\mathcal{V}}}(P(1))=\rad t_{\mathcal{V}}(P(1))=\rad (P(1))=S(1)\oplus S(m+1)\oplus S(m+2)\oplus T, $$
\xymatrix@-1.0pc@C=0.05pt {
& 2\edge[d]             \\
{\rm where}  \;\;\;T=  &3\edge[d]\\
&\vdots\edge[d] \\
& \;m-1.}

Since $\top S(1)=S(1)\in \add \mathcal{V}'$, we have $t_{\mathcal{V}}(S(1))=S(1)$ by Proposition \ref{prop3.1}(3).
Similarly, $t_{\mathcal{V}}(S(m+1))=S(m+1)$, $t_{\mathcal{V}}(S(m+2))=S(m+2)$ and $t_{\mathcal{V}}(T)=T$.
So $$t_{\mathcal{V}}F_{t_{\mathcal{V}}}(P(1))=t_{\mathcal{V}}(S(1)\oplus S(m+1)\oplus S(m+2)\oplus T)
=S(1)\oplus S(m+1)\oplus S(m+2)\oplus T,$$
and hence $$F_{t_{\mathcal{V}}}^{2}(P(1))=\rad t_{\mathcal{V}}F_{t_{\mathcal{V}}}(P(1))
=\rad(S(1)\oplus S(m+1)\oplus S(m+2)\oplus T)=\rad T.$$
It is easy to see that $\rad T\in \mathfrak{F}(\mathcal{V})$, so $t_{\mathcal{V}}(\rad T)=0$ by Proposition \ref{prop3.1}(1).
Moreover, $t_{\mathcal{V}}F_{t_{\mathcal{V}}}^{2}(P(1))=0$. It follows that $\ell\ell^{t_{\mathcal{V}}}(P(1))=2$.
Similarly, we have
\begin{equation*}
\ell\ell^{t_{\mathcal{V}}}(P(i))=
\begin{cases}
2, &\text{if}\;\;i=2;\\
1,&\text{if} \;\;3\leqslant  i\leqslant m+2.
\end{cases}
\end{equation*}
Thus $n:=\ell\ell^{t_{\mathcal{V}}}(\Lambda)=\max\{\ell\ell^{t_{\mathcal{V}}}(P(i)) \;|\; 1 \leqslant i  \leqslant m+2\}=2$.

(1) Because $\LL(\Lambda)=m-1$, we have
$$\dim D^{b}(\mod \Lambda) \leqslant\LL(\Lambda)-1=m-2$$ by Theorem \ref{thm1.1}(1). In particular,
from Theorem \ref{thm1.1}(2), we can not get an upper bound for $\dim D^{b}(\mod \Lambda)$.
By Theorem \ref{thm1.1}(3), we have
$$\dim D_{sg}^{b}(\mod \Lambda)\leqslant \LL(\Lambda)-2=m-3.$$

(2) By Theorem \ref{thm3.12}, we have
$$\dim D^{b}(\mod \Lambda) \leqslant (d+2)(n+1)-2=7.$$
By Theorem \ref{thm3.14}, we have
$$\dim D_{sg}^{b}(\mod \Lambda)=0.$$}
\end{example}

\begin{example}\label{exa4.2}
{\rm Consider the bound quiver algebra $\Lambda=kQ/I$, where $k$ is an algebraically closed field and $Q$
is given by
$$\xymatrix{
&1\ar[r]^{\alpha_{1}} \ar[d]^{\alpha_{m+1}} &2\ar[r]^{\alpha_{2}} &{3}\ar[r]^{\alpha_{3}}&\cdots \ar[r]^{\alpha_{m-1}}&{m}\\
&m+1\ar[r]^{\alpha_{m+2}}
&m+2\ar[r]^{\alpha_{m+3}}&m+3\ar[r]^{\alpha_{m+4}}&\cdots\ar[r]^(0.4){\alpha_{2m-1}}&2m-1
}$$
and $I$ is generated by
$\{\alpha_{i}\alpha_{i+1}\;| \;m+1\leqslant i\leqslant 2m-2\}$ with $m\geqslant 9$.
Then the indecomposable projective $\Lambda$-modules are
$$\xymatrix@-1.0pc@C=0.1pt
{&&  1\edge[d]\edge[ld] &&&& & 2\edge[d] &&&& & &&&& &&&& &\\
&m+1  &2\edge[d] &&&& &3\edge[d]  &&&& & 3\edge[d]  &&&&&j\edge[d] &&&& &\\
P(1)=& &3\edge[d] &&&&P(2)=&4\edge[d] &&&&P(3)=&4\edge[d]&&&&P(j)=&\;j+1,&&&&P(2m-1)=2m-1,&\\
&  &\vdots\edge[d]&&&& &\vdots\edge[d]&&&& &\vdots\edge[d]&&&& &&&& &\\
&  &\;m,  &&&& &\;m, &&&& &\;m, &&&&&&&& &\\
&  &  &&& & &&&&  &  &&&& &&&& &\\
}$$
where $m+1 \leqslant j\leqslant 2m-2$ and $P(i+1)=\rad P(i)$ for any $2 \leqslant i\leqslant m-1$; and
the indecomposable injective $\Lambda$-modules are
$$ \xymatrix@-1.0pc@C=0.1pt
{& & 1\edge[d]&&&& & 1\edge[d] &&&&  &  &&&&  &  &&&&  &  \\
&  &  2\edge[d] &&&&  & 2\edge[d] &&&& & 1\edge[d]&&&& &j-1\edge[d]&&&&  &\\
&  I(m)=&\vdots\edge[d]  &&&&I(m-1)= & \vdots\edge[d]  &&&& I(m+1)=&\;m+1, &&&&I(j)=&\;j, &&&&&\\
&  &  \;m,  & &&&& \;m-1,  &  &&&& & &&&& & &&&& &\\
&  &  &  &&&& & &&&&  &  &&&& & &&&&  &\\}$$
where $m+2\leqslant j\leqslant 2m-1$ and $I(i)=I(i+1)/\soc I(i+1)$ for any $1 \leqslant i\leqslant m-1$.

We have
\begin{equation*}
\pd S(i)=
\begin{cases}
m-1, &\text{if}\;\;i=1;\\
1,&\text{if} \;\;2 \leqslant  i\leqslant m-1;\\
0,&\text{if} \;\; i=m;\\
2m-1-i,&\text{if}\;\; m+1 \leqslant  i\leqslant 2m-1,
\end{cases}
\end{equation*}
and $\mathcal{S}^{<\infty}=\mathcal{S}$. We also have
\begin{equation*}
\id S(i)=
\begin{cases}
0, &\text{if}\;\;i=1;\\
1,&\text{if} \;\;2\leqslant  i\leqslant m;\\
i-m,&\text{if} \;\;m+1\leqslant  i\leqslant  2m-1.
\end{cases}
\end{equation*}
Let $\mathcal{V}:=\{S(i)\mid 2\leqslant i \leqslant m\}\subseteq\mathcal{S}^{<\infty}$. Then
$$a:=\pd\mathcal{V}=1,\;c:=\id\mathcal{V}=1\; {\rm and}\; d:=\min\{ a, c\}=1.$$
Let $\mathcal{V}'$ be all the others simple modules in $\mod \Lambda$, that is,
$\mathcal{V}'=\{S(i)\mid i=1\ \text{or}\ m+1 \leqslant i\leqslant 2m-1\}$.
Similar to the computation in Example \ref{exa4.1}, we have $n:=\ell\ell^{t_{\mathcal{V}}}(\Lambda)=2$.

(1) Because $\LL(\Lambda)=m$, we have
$$\dim D^{b}(\mod \Lambda) \leqslant\LL(\Lambda)-1=m-1$$ by Theorem \ref{thm1.1}(1).
Because $\gldim \Lambda=m-1$, we also have
$$\dim D^{b}(\mod \Lambda) \leqslant\gldim \Lambda=m-1$$ by Theorem \ref{thm1.1}(2). In addition, we have
$$\dim D^b_{sg}(\mod \Lambda) \leqslant\LL(\Lambda)-2=m-2$$ by Theorem \ref{thm1.1}(3).

(2) By Theorem \ref{thm3.12}, we have
$$\dim D^{b}(\mod \Lambda) \leqslant (d+2)(n+1)-2=7.$$ By Theorem \ref{thm3.14}, we have
$$\dim D_{sg}^b(\mod \Lambda)=0.$$}
\end{example}

In the above two examples, the upper bounds in (2) are smaller than that in (1) and the difference
between them may be arbitrarily large.

\vspace{0.5cm}

{\bf Acknowledgements.}
This research was partially supported by National Natural Science Foundation of China
(Grant Nos. 11971225, 11571164) and a Project
Funded by the Priority Academic Program Development of Jiangsu Higher Education
Institutions. The authors would like to thank Dong Yang for his helpful discussions,
and thank the referees for very useful and detailed suggestions.

\end{document}